\documentclass{article}

\usepackage{amsbsy,amscd,amsfonts,amsmath,amsbsy,mathrsfs}
\usepackage{tikz-cd}
\usepackage{array}
\usepackage{multirow}
\usepackage{graphics}
\usepackage{enumerate}
\usepackage{amsthm}
\usepackage{hyperref}
\usepackage{subcaption}
\usepackage{float}
\usepackage{verbatim}
\usepackage[all]{xy}

\begin{document}

\captionsetup[figure]{labelformat=empty}

\theoremstyle{plain}
\newtheorem{lemma}{Lemma}[section]
\newtheorem{theorem}[lemma]{Theorem}
\newtheorem{proposition}[lemma]{Proposition}
\newtheorem{corollary}[lemma]{Corollary}
\theoremstyle{definition}
\newtheorem{remark}[lemma]{Remark}
\title{On the classification of $S^3$-bundles over $\mathbb{C}P^2$}
\author{Liu wancheng}
\date{\today}
\maketitle

\begin{abstract}
    This paper presents a classification of the total spaces of $S^3$-bundles over $\mathbb{C}P^2$ up to orientation-preserving homotopy equivalence. Our approach proceeds in two steps: we first derive the PL-homeomorphism classification for these manifolds by computing their Kreck-Stolz invariants. Then, building upon this PL classification result and through an application of surgery theory, we establish the homotopy equivalence classification.
\end{abstract}

\section{Introduction}

Examples of nondiffeomorphic homeomorphic homogenuous spaces were first discovered by Kreck and Stolz  in \cite{kreck}. These examples belong to 
a more general class of manifolds called manifolds of type $r$. A closed 1-connected smooth $7$-dimensional manifold  $M$ is 
called a manifold of type $r$ if it satisfies the following properties:
\begin{enumerate}
    \item $H^0(M;\mathbb{Z})\cong H^2(M;\mathbb{Z})\cong H^5(M;\mathbb{Z})\cong H^7(M;\mathbb{Z})\cong \mathbb{Z}$
    \item $H^1(M;\mathbb{Z})\cong H^3(M;\mathbb{Z})\cong H^6(M;\mathbb{Z})=0$
    \item $H^4(M;\mathbb{Z})\cong \mathbb{Z}_r$
    \item $H^4(M;\mathbb{Z})$ is generated by $u^2$, where $u$ is a generator of $H^2(M;\mathbb{Z})$.
\end{enumerate} 
In \cite{kreck_2}\cite{kreck1998correction} Kreck and Stolz introduced invariants $s_i(\bar{s}_i)$ that provide a complete smooth (topological) classification of manifolds of type $r$.
However, the classification of manifolds of type $r$ up to homotopy equivalence remains unsolved. Partial results can be found in Kruggel's paper \cite{kruggel_2}.

The total spaces of $S^3$-bundles over $\mathbb{C}P^2$ with Euler class $e=\pm r$ are typical examples of manifolds of type $r$.
Escher  \cite{escher2014topology} obtained a classification of these manifolds up to diffeomorphism or homeomorphism using Kreck and Stolz's invariants. A partial homotopy classification based on Kruggel's analysis was also presented in the same paper. 
In this paper, we give a complete classification of these manifolds up to homotopy equivalence (Theorem \ref{homotopy}). The proof is based on the PL-homeomorphism classification and does not depend on  Kruggel's earlier work.

\emph{Convention}. All manifolds under consideration
are oriented, and all diffeomorphisms, homeomorphisms, PL-homeomorphisms and homotopy equivalences are orientation-preserving.

\subsection{Rank $4$ real vector bundles over $\mathbb{C}P^2$}

In this subsection we give a brief description of $4$-dimensional oriented real vector bundles over $\mathbb C P^2$ and fix the notations. 

Applying the functor $[-, BSO(4)]$ to the cofiber sequence 
\begin{eqnarray}\label{seq:1}
S^3 \stackrel{h}{\longrightarrow} S^2 \longrightarrow \mathbb{C}P^2 \stackrel{q}{\longrightarrow} S^4
\end{eqnarray} 
associated to the Hopf map $h \colon S^3 \to S^2$, and notice that $\pi_3(BSO(4))=0$, one has a short  sequence
\begin{eqnarray}\label{seq:2}
0\rightarrow \pi_4(BSO(4)) \rightarrow \left[\mathbb{C}P^2,BSO(4)\right] \rightarrow \pi_2(BSO(4))\rightarrow 0.
\end{eqnarray}
We choose orientation cohomology class $w_{\mathbb{C}P^2}$ of 
$\mathbb{C}P^2$,  and $w_{S^4}$ of $S^4$, respectively,  such that $q^*w_{S^4}=w_{\mathbb{C}P^2}=x^2$, where $x$ is a generator of $H^2(\mathbb{C}P^2;\mathbb{Z})$. 

Let $e$ and $p_1$ denote the Euler class and the first Pontrjagin class of a vector bundle, respectively.  We choose generators $ \alpha$ and $\beta$ of the group $\pi_4(BSO(4))$, which is known to be isomorphic to $\mathbb{Z} \oplus\mathbb{Z}$, such that the corresponding vector bundles over $S^4$ satisfy the following properties:

\begin{center}
\begin{tabular}{c|c|c} 
&$\alpha$ & $\beta$ \\ 
\hline 
$e$ & 0 & $\omega_{S^4}$ \\ 
\hline 
$p_1$ & $4\omega_{S^4}$ & $-2\omega_{S^4}$ \\ 
\end{tabular} 
\end{center}

Under the isomorphisms $\pi_2(BSO(4)) \cong \mathbb{Z}_2 \cong H^2(\mathbb C P^2;\mathbb Z_2)$, the homomorphism $[\mathbb{C}P^2,BSO(4)] \to \pi_2(BSO(4))$ is identified with the homomorphism by taking the second Stiefel-Whitney class $w_2$ of a vector bundle. 

Therefore, from the exact sequence (\ref{seq:2}),  one may deduce that $4$-dimensional oriented real bundles over $\mathbb{C}P^2$ can be obtained by the following two pullback diagrams and parametrized by a pair of integers $(k,l)$ in each case.

\begin{figure*}[htbp]
\centering
\captionsetup[subfigure]{labelformat=empty}
\begin{subfigure}{0.45\textwidth}
\centering
\begin{tikzcd}
\xi_{k,l}\arrow[r]\arrow[d] & k\alpha +l\beta \arrow[d]\\
\mathbb{C}P^2\arrow[r,"q"] & S^4
\end{tikzcd}
  \caption{Spin cases}
\end{subfigure}
\hfill
\begin{subfigure}{0.45\textwidth}
  \centering
    \begin{tikzcd}
    \xi_{k,l}^{\prime}\arrow[r,]\arrow[d] & \gamma\oplus\varepsilon^2\vee k\alpha +l\beta \arrow[d]\\
    \mathbb{C}P^2\arrow[r,"p"] & \mathbb{C}P^2\vee S^4
    \end{tikzcd}
  \caption{Nonspin cases}
\end{subfigure}
\end{figure*}
\noindent 
Here $\gamma$ is the tautological line bundle over $\mathbb{C}P^2$, $p$ is the map collapsing the boundary of a closed disc embeded in the interior of the $4$-cell of $\mathbb{C}P^2$ to a point. 
An easy calculation shows  

\begin{equation}\label{cc}
\begin{aligned}
    p_1(\xi_{k,l})=(4k-2l)w_{\mathbb{C}P^2}, &\ e(\xi_{k,l})=lw_{\mathbb{C}P^2},\\ 
    p_1(\xi_{k,l}^{\prime})=(4k-2l+1)w_{\mathbb{C}P^2}, &\ e(\xi_{k,l})=lw_{\mathbb{C}P^2}.
\end{aligned}
\end{equation}

\subsection{Classification Results}
Let $M_{k,l}$ and $M_{k,l}^{\prime}$ denote the total spaces of the  sphere bundles associated with $\xi_{k,l}$ and $\xi_{k,l}^{\prime}$, respectively. Noticing that $\mathbb C P^2$ is nonspin, a straightforward calculation shows that $M_{k,l}$ is nonspin and $M_{k,l}^{\prime}$ is spin.
We are interested in the classification of these natural geometric objects up to diffeomorphism, homeomorphism or homotopy equivalence. 

When $l \ne 0$, diffeomorphism and homeomorphism classifications are obtained by applying Kreck-Stolz invariants \cite{kreck_2}. 
Similar methods will produce a PL-homeomorphism classification (see \S \ref{tool}). We sum up the results in the following theorem.

\begin{theorem}\label{result}$\ $
    \begin{enumerate}
        \item $M_{k,l}$ is diffeomorphic to $M_{k^{\prime},l}$ if and only if $k$ and $k^{\prime}$ satisfy the following:
        \begin{enumerate}
            \item $k \equiv k^{\prime} \pmod{6l}$,
            \item $(k-k^{\prime})(3(k+k^{\prime}-l)+1) \equiv 0 \pmod{168l}$.
        \end{enumerate}
        \item $M^{\prime}_{k,l}$ is diffeomorphic to $M^{\prime}_{k^{\prime},l}$ if and only if $k$ and $k^{\prime}$ satisfy the following:
        \begin{enumerate}
            \item $k\equiv k^{\prime}\pmod{12l}$,
            \item $(k-k^{\prime})(k+k^{\prime}-l+2)\equiv 0 \pmod{56l}$.
        \end{enumerate}
        \item $M_{k,l}$ is homeomorphic to $M_{k^{\prime},l}$ if and only if $k\equiv k^{\prime}\pmod{6l}$. 

        \item $M^{\prime}_{k,l}$ is homeomorphic to $M^{\prime}_{k^{\prime},l}$ if and only if $k\equiv k^{\prime}\pmod{ 12l}$.
        \item  $M_{k,l}$ is PL-homeomorphic to $M_{k^{\prime},l}$ if and only if $k\equiv k^{\prime}\pmod{ 6l}$,
        \item $M^{\prime}_{k,l}$ is PL-homeomorphic to $M^{\prime}_{k^{\prime},l}$ if and only if $k\equiv k^{\prime}\pmod{ 12l}$.
    \end{enumerate} 
\end{theorem}

Escher \cite{escher2014topology} calculated the Kreck-Stolz's invariants of these manifolds and obtained the diffeomorphism and homeomorphism classifications. 

The main result of this paper is a homotopy classification of the total spaces of $S^3$-bundles over $\mathbb C P^2$ with non-vanishing Euler classes.

\begin{theorem}\label{homotopy}
    $M_{k,l}$ is homotopy equivalent to $M_{k^{\prime},l}$ if and only if $k\equiv k^{\prime}$ $\pmod{ 6}$. 
     The same holds when $M_{k,l}$ and $M_{k^{\prime},l}$ are replaced by $M^{\prime}_{k,l}$ and $M^{\prime}_{k^{\prime},l}$, respectively. 
\end{theorem}

From the exact sequence associated to the fiber bundle $S^3 \to M'_{k,l} \to \mathbb C P^2$
$$\cdots \to \pi_4(S^3) \to \pi_4(M'_{k,l}) \to \pi_4(\mathbb C P^2) \to \cdots$$
one  sees that $\pi_4(M'_{k,l})$ is either isomorphic to $\mathbb Z_2$ or vanishes.
It is shown in \cite{escher2014topology} that  $\pi_4(M^{\prime}_{k,l})$ is isomorphic to $\mathbb{Z}_2$ when $l$ is even, and conjectured that $\pi_4(M^{\prime}_{k,l})$ vanishes when $l$ is odd, as suggested by calculations for specific examples. As a corollary of Theorem \ref{homotopy} we confirmed this conjecture. 

\begin{corollary}\label{cor:hmtpgp}
The homotopy group $\pi_4(M^{\prime}_{k,l})$ vanishes when $l$ is odd.
\end{corollary}

We will prove Theorem \ref{homotopy} in  Section \ref{3} using the PL-homeomorphism classification in Theorem \ref{result} and surgery theory.

When $l=0$, the cohomology rings of both $M_{k,l}$ and $M_{k,l}^{\prime}$ are isomorphic to $H^*(\mathbb{C}P^2\times S^3)$. The classification
results are obtained in \cite{wxq}. For completeness we present the results here.

\begin{theorem}[Theorem 3.1 of \cite{wxq}]
\

    \begin{enumerate}[1]
        \item  Two manifolds $M_{k,0}$ and $M_{k^{\prime},0}$ are homeomorphic, PL-homeomorphic or diffeomorphic if and only if $k=k^{\prime}$;
        \item  Two manifolds $M_{k,0}$ and $M_{k^{\prime},0}$ are homotopy equivalent if and only if $k\equiv k^{\prime} \pmod{ 6}$.
    \end{enumerate}

    These two statements also holds when $M_{k,0}$ and $M_{k^{\prime},0}$ are replaced by $M^{\prime}_{k,0}$ and $M^{\prime}_{k^{\prime},0}$, respectively.

\end{theorem}
\vspace{2pt}

\section{Kreck-Stolz Invariants and PL Classification}\label{tool}
In the remaining part of this paper we always assume that  $l$ is nonzero. 
The smooth and homeomorphism classifications of the manifolds of type $l$ are obtained by Kreck-Stolz in \cite{kreck_2}. In this section we give a summary of the invariants which they used to achieved the classification, and show that after obvious modifications, the same approach also applies to the PL classification problem. 

Suppose $(M,u)$ is a  manifold of 
type $l$, the invariants are defined as certain characteristic numbers of a pair $(W,z)$ whose boundary is $(M,u)$. That is, we
have the following:

\begin{enumerate}
    \item $W$ is a smooth $8$-manifold with $\partial W=M$;
    \item $z\in H^2(W,\mathbb{Z})$ such that $z|_M =u$;
    \item $W$ is spin (nonspin) if $M$ is spin (nonspin);
    \item In the nonspin case, $w_2(W) \equiv z \pmod{ 2}$.
\end{enumerate}

Notice that $j^*:H^4(W,M;\mathbb{Q}) \to H^4(W;\mathbb{Q})$ is an isomorphism. Therefore the cohomology classes $p_1(W)^2$, $p_1(W) z^2$ and $z^4$ can be regarded as elements in $H^8(W,M;\mathbb Q)$, and we abbreviate $p_1^2$, $p_1z^2$ and $z^4$ for the evaluations $\langle p_1(W)^2, [W,M] \rangle$, $\langle p_1(W) z^2, [W,M] \rangle$ and $\langle z^4, [W,M] \rangle$, respectively. The charcteristic numbers $S_i(W,z)\in \mathbb{Q}$ are defined as follows:
\begin{align*}
    \text{Spin}&\ \text{case:}\\
    & S_1(W,z)=-\frac{1}{2^5\cdot 7}\text{sign}(W)+\frac{1}{2^7\cdot7}p_1^2,\\
            & S_2(W,z)=-\frac{1}{2^4\cdot3}z^2p_1+\frac{1}{2^3\cdot3}z^4,\\
            &   S_3(W,z)=-\frac{1}{2^2\cdot3}z^2p_1+\frac{2}{3}z^4\\
    \text{Nonspin}&\ \text{case:}\\
        & S_1(W,z)=-\frac{1}{2^5\cdot7}\text{sign}(W)+\frac{1}{2^7\cdot7}p_1^2-\frac{1}{2^6\cdot3}z^2p_1+\frac{1}{2^7\cdot3}z^4\\
        &S_2(W,z)=-\frac{1}{2^3\cdot3}z^2p_1+\frac{5}{2^3\cdot3}z^4\\
       &S_3(W,z)=-\frac{1}{2^3}z^2p_1+\frac{13}{2^3}z^4
\end{align*}

Define $\bar{S_i}(W,z)$ as follows: $\bar{S_1}(W,z)=28S_1(W,z)$, $\bar{S_2}(W,z)=S_1(W,z)$, and $\bar{S_3}(W,z)=S_3(W,z)$.

Define $s_i(M)\in \mathbb{Q}/\mathbb{Z}$ by $S_i(W,z)\pmod{\mathbb{Z}}$, and $\bar{s_i}(M)\in \mathbb{Q}/\mathbb{Z}$ by $\bar{S_i}(W,z)\pmod{\mathbb{Z}}$.
Kreck and Stolz's classification of manifolds of type $l$ is as follows.
\begin{theorem}[Theorem 1 of \cite{kreck1998correction}]\label{smoothtop}
    Let $M$ and $M^{\prime}$ be two smooth manifolds of type $l$ which are both spin or both nonspin. Then $M$ is diffeomorphic (hemeomorphic) to 
    $M^{\prime}$ if and only if $s_i(M)=s_i(M^{\prime})$ (resp. $\bar{s_i}(M)=\bar{s_i}(M^{\prime})$) for $i=1,\ 2,\ 3.$
    \vspace*{2pt}
\end{theorem}

By smoothing theory, if a closed oriented $7$-dimensional PL manifold $M$ adimits a smoothing, then it adimits exactly $28$ smoothings up to concordance \cite{smoothing}.
Moreover, every concordant class can be represented by $M\# \Sigma$, where $\Sigma$ is an exotic sphere. 
Therefore, we could know $M^{\prime}$ is $PL$-homeomorphic to $M$ if and only if $M^{\prime}$ is diffeomorphic to $M\# \Sigma$ for some exotic
sphere $\Sigma$.
Consider 8-dimensional Milnor manifold $M^8(E_8)$ which is a smooth parallelizable manifold with signature $8$. 
$\partial M^8(E_8)$ is a generator of $\Theta_7$, say $\Sigma_1$. We have $s_1(\Sigma_1)=1/28$, $s_2(\Sigma_1)=s_3(\Sigma_1)=0$. 
Therefore, we have the following
theorem:
\begin{theorem}\label{PL}
    Let $M$ and $M^{\prime}$ be two smooth manifolds of type $l$ which are both spin or both nonspin. Then $M$ is PL-homeomorphic to 
    $M^{\prime}$ if and only if $28s_1(M)=28s_1(M^{\prime})$, and  $s_i(M)=s_i(M^{\prime})$ for $i=2,\ 3.$
    \vspace*{2pt}
\end{theorem}


\begin{proof}[Proof of Theorem \ref{result}]
    Let $W_{k,l}$ be the disk bundle associated to $\xi_{k,l}$ , and let $\pi_{k,l}: W_{k,l}\rightarrow \mathbb{C}P^2$ be the projection.
    We have $\tau(W_{k,l})\cong \pi_{k,l}^*(\tau(\mathbb{C}P^2)\oplus \xi_{k,l})$.  
    Let $z=\pi_{k,l}^*x$, recall that $x$ is a generator of $H^2(\mathbb{C}P^2;\mathbb{Z})$, then 
    \begin{align*}
        p_1(W_{k,l})=&\pi_{k,l}^*(p_1(\mathbb{C}P^2)+p_1(\xi_{k,l}))\\
            =&(4k-2l+3)\pi_{k,l}^*w_{\mathbb{C}P^2}\\
            =&(4k-2l+3)z^2.
    \end{align*}

    Let $y$ be a generator of $H^4(W_{k,l},M_{k,l};\mathbb{Z})$ such that $j^*(y)=lz^2$. 
    Clearly, $y$ is the Thom class of $\xi_{k,l}$.
    The orientation of $W_{k,l}$ is defined by equation 
    $\langle z^2\cup y, \left[W_{k,l},M_{k,l}\right] \rangle =1$.

    Since $\langle y\cup j^*y, \left[W_{k,l},M_{k,l}\right]\rangle = l$, sign($W$) equals to the sign of $l$.
    By definition, we have
    \begin{align*}
        z^4 &= \langle z^2\cup j^{*-1}z^2, \left[W_{k,l},M_{k,l}\right] \rangle = \frac{1}{l}\langle z^2\cup y, \left[W_{k,l},M_{k,l}\right] \rangle = \frac{1}{l},\\
        z^2p_1 &= \frac{4k-2l+3}{l},\\
        p_1^2 &= \frac{(4k-2l+3)^2}{l}.
    \end{align*} 
    Then $s_i$ is given by
    \begin{align*}
        s_1 &= -\frac{\text{sign}(l)}{2^5\cdot7}+\frac{(4k-2l+3)^2}{2^7\cdot7l}-\frac{8k-4l+5}{2^7\cdot3l},\\
        s_2 &= -\frac{2k-l-1}{12l},\\
        s_3 &= -\frac{2k-l-5}{4l}.
    \end{align*}

    Similarly, let $\pi^{\prime}_{k,l}: W_{k,l}^{\prime}\rightarrow \mathbb{C}P^2$ be the projection. We have 
    \begin{align*}
        z^4 &= \frac{1}{l},\\
        z^2p_1 &= \frac{4k-2l+4}{l},\\
        p_1^2 &= \frac{(4k-2l+4)^2}{l}.
    \end{align*}
    And $s_i$ is given by
    \begin{align*}
        s_1 &= -\frac{\text{sign}(l)}{2^5\cdot7}+\frac{(2k-l+2)^2}{2^5\cdot7l},\\
        s_2 &= -\frac{2k-l+1}{24l},\\
        s_3 &= -\frac{2k-l-2}{6l}.
    \end{align*}

   Then Theorem \ref{result} follows by comparing the invariants.
\end{proof}

\begin{proof}[Proof of Corollary \ref{cor:hmtpgp}]
Let $l$ be an odd integer. If $\pi_4(M'_{k,l}) = \mathbb Z_2$ for some $k$, by a theorem of Kruggel \cite[Theorem 0.1]{kruggel_2}, the $s_2$-invariant of any type $l$ manifold $M$ which is homotopy equivalent to $M'_{k,l}$ satisfies  $l(s_2(M)-s_2(M'_{k,l}))=0 \in \mathbb Q /\mathbb Z$.  By Theorem \ref{homotopy}, $M'_{k+6,l}$ is homotopy equivalent to $M_{k,l}$. But the above calculation shows that $l(s_2(M'_{k+6,l})-s_2(M'_{k,l}))=1/2 \in \mathbb Q / \mathbb Z$, a contradiction.
\end{proof}

\section{Fiber homotopy equivalences}\label{3}

\begin{lemma}\label{fiber}
    There exist fiber homotopy equivalences of spherical fibrations $f_j:M_{k+6j,l}\rightarrow M_{k,l}$ and $f_j^{\prime}:M_{k+6j,l}^{\prime}\rightarrow M_{k,l}^{\prime}$
    for all $j\in \mathbb{Z}$.
\end{lemma}
\begin{proof}
Let $SG(4)$ be the topological monoid of orientation preserving self homotopy equivalences of $S^3$, $BSG(4)$ the classifying space of $S^3$-fibrations. Therefore, the sphere bundles of two vector bundles are fiber homotopy equivalent as $S^3$-fibrations if and only if the classifying maps of the vector bundles have the same image under the map $i_* \colon \left[\mathbb{C}P^2,BSO(4)\right] \to \left[\mathbb{C}P^2,BSG(4)\right]$ induced by the inclusion $i:SO(4)\rightarrow SG(4)$.

The cofibration sequence (\ref{seq:1}) induces the following exact diagram

\tikzset{global scale/.style={
    scale=#1,
    every node/.append style={scale=#1}
    }
    }
\setcounter{figure}{0}  
\begin{figure}[htbp]
\centering
\begin{tikzcd}[global scale=0.91,column sep=small]
0   \ar[r]  & \pi_4(BSO(4)) \ar[r,"q^*"]\ar[d,"i_*"]  &\left[\mathbb{C}P^2,BSO(4)\right] \ar[r]\ar[d,"i_*"]   &\pi_2(BSO(4))  \ar[r]\ar[d,"i_*"]  &0\\
\pi_3(BSG(4)) \ar[r,"(\sum h)^*"]    &   \pi_4(BSG(4)) \ar[r,"q^*"]  &\left[\mathbb{C}P^2,BSG(4)\right] \ar[r]   &\pi_2(BSG(4))  \ar[r,"h^*"]  &\pi_3(BSG(4))
\end{tikzcd}
\caption{Diagram 1}
\label{diagram_1}
\end{figure}

It's known that $i_*:\pi_4(BSO(4))\rightarrow \pi_4(BSG(4))$ can be identified with 
$\mathbb{Z}\{\alpha\}\oplus\mathbb{Z}\{\beta\}\rightarrow \mathbb{Z}_{12}\oplus\mathbb{Z}$, $k\alpha+l\beta \mapsto ((k\bmod{12}),l)$ \cite{crowley}. To identify the homomorphism $(\sum h)^*: \pi_3(BSG(4))\rightarrow \pi_4(BSG(4))$ we consider the fibration $SF(3)\rightarrow SG(4) \rightarrow S^3$, where $ SG(4) \to S^3$ is the evaluation at the base point, and  $SF(3)$ the subspace of $SG(4)$ consisting of orientation preserving homotopy equivalences of $S^3$ which fix the base point. 
The Hopf map $h \colon S^3 \to S^2$ induces the following diagram
\begin{figure}[H]
\centering
\begin{tikzcd}
    0   \ar[r]  & \pi_2(SF(3)) \ar[r,"\cong"]\dar["h^*"] &\pi_2(SG(4)) \ar[r]\dar["h^*"]   & \pi_2(S^3) \rar\dar["h^*"]   & 0\\
    0   \ar[r]  & \pi_3(SF(3)) \ar[r]        &\pi_3(SG(4)) \ar[r]          & \pi_3(S^3) \rar & 0
\end{tikzcd}
\caption{Diagram 2} 
\label{diagram_2}
\end{figure}

By the usual adjoint correspondence, $h^*:\pi_2(SF(3))\rightarrow \pi_3(SF(3))$ can be identified with $(\Sigma^3h)^*:\pi_5(S^3)\rightarrow \pi_6(S^3)$, which is known to be identified with the inclusion 
$\mathbb{Z}_2\hookrightarrow \mathbb{Z}_{12}$ \cite{Toda}.
Therefore, $(\sum h)^*: \pi_3(BSG(4))\rightarrow \pi_4(BSG(4))$ 
can be identified with the inclusion into the first summand $\mathbb{Z}_2 \hookrightarrow \mathbb{Z}_{12}\oplus\mathbb{Z}$.

Therefore from commutative diagram \ref{diagram_1}
one has:
$$i_*(\xi_{k+6j,l})=q^*i_*((k+6j)\alpha+l\beta)=q^*i_*(k\alpha+l\beta)=i_*(\xi_{k,l}).$$ 
This shows that $M_{k+6j,l}$ and$M_{k,l}$ are fiber homotopy equivalent.

For the nonspin case one needs more detailed analysis since the exact sequences in diagram \ref{diagram_1} are not exact sequences of abelian groups. 
We will prove the statement in two steps. 

First we show that the spherical fibrations $\xi_{k+6j,l}'$ and $\xi_{k,l}'$ are stably fiber homotopy equivalent. 
The cofiber sequence (\ref{seq:1}) induces the following commutative diagram of short exact sequences

\begin{center}

\begin{tikzcd}[global scale=0.95,column sep=small]
        0   \ar[r]  & \pi_4(BSO) \ar[r]\ar[d,two heads]  &\left[\mathbb{C}P^2,BSO\right] \ar[r]\ar[d,"i_*"]   &\pi_2(BSO)  \ar[r]\ar[d,"\cong"]  &0\\
        \pi_3(BSG) \ar[r,"\sum\eta^*"]    &   \pi_4(BSG) \ar[r,"q^*"]  &\left[\mathbb{C}P^2,BSG\right] \ar[r]   &\pi_2(BSG)  \ar[r,"0"]  &\pi_3(BSG)
\end{tikzcd}
    
\end{center}
The following facts are known. 
\begin{enumerate}
    \item Taking the first Pontrjagin class induces an bijective homomorphism  
    $$p_1 \colon \left[\mathbb{C}P^2,BSO\right] \to \mathbb{Z}.$$
    \item There are isomorphisms $\pi_4BSG\cong \pi_3^s\cong \mathbb{Z}_{24}$, and we have the following diagram
    
    \begin{center}

    \begin{tikzcd}
        \pi_3(BSG(4))\cong \mathbb{Z}_2 \rar\dar & \pi_4(BSG(4))\cong \mathbb{Z}_{12}\oplus\mathbb{Z}\dar\\
        \pi_3(BSG) \cong \mathbb{Z}_2 \rar["(\sum\eta)^*",hook] & \pi_4(BSG)\cong \mathbb{Z}_{24}
    \end{tikzcd}
    \end{center}

    where the natural map $\pi_4(BSG(4))\rightarrow \pi_4(BSG)$ is the injection on the first summand and zero on the second summand,
    as $\Sigma^{\infty}:\pi_6(S^3)\rightarrow \pi_3^s$ is indentified with $\mathbb{Z}_{12} \hookrightarrow \mathbb{Z}_{24}$ \cite{Toda}.
\end{enumerate}
Therefore by the $5$-lemma the  homomorphism $i_*:\left[\mathbb{C}P^2,BSO\right]\rightarrow \left[\mathbb{C}P^2,BSG\right]$ is surjective . 
Hence the group $\left[\mathbb{C}P^2,BSG\right]$ is cyclic and isomorphic to $ \mathbb{Z}_{24}$ by counting order. 
Under stabilization $BSO(4) \to BSO \to BSG$ the classifying maps for the vector bundles $ \xi_{k,l}^{\prime}$ and $\xi_{k+6j,l}^{\prime}$ are homotopic, since $p_1(\xi_{k+6j,l}^{\prime}) \equiv p_1(\xi_{k,l}^{\prime}) \pmod{24}$ (see formula (\ref{cc})). 
Therefore there exists a homotopy 
$$f_t \colon \mathbb C P^2 \times [0,1] \to BSG$$ 
between the classifying maps for $i_*j_*(\xi_{k+6j,l}^{\prime})$ and $i_*j_*(\xi_{k,l}^{\prime})$. 

In the second step we find by induction on skeleta a  lift of $f_t$ to  $BSG(4)$ 
$$g_t:\mathbb{C}P^2\times I\rightarrow BSG(4),$$ 
with $g_0 =i_*\xi_{k+6j,l}^{\prime}$ and $g_1 =i_*\xi_{k,l}^{\prime}$. This  shows that $\xi_{k+6j,l}^{\prime}$ and $\xi_{k,l}^{\prime}$ are fiber homotopy equivalent.

Let $SG/SG(4)$ be the homotopy fibre of $BSG(4)\rightarrow BSG$, then a basic calculation shows that $\pi_2SG/SG(4)= 0$ and $\pi_4SG/SG(4)\cong \mathbb{Z}$. By elementary obstruction theory there is a lift $g_t^{(2)}$  of $f_t$ on the $2$-skeleton of $\mathbb C P^4$
$$g_t^{(2)} \colon S^2 \times I \to BSG(4).$$
Let $X=\mathbb{C}P^2\times\partial I\cup S^2\times I$, then $\mathbb{C}P^2\times I= X\cup_m D^5$, where we regard $S^4 = D^4 \times \partial I \cup S^3 \times I$ and $m \colon S^4 \to X$ has the form 
$$\iota \times \mathrm{id} \cup h \times \mathrm{id} \colon  D^4 \times \partial I \cup S^3 \times I \to \mathbb C P^2 \times I$$
where $\iota \colon D^4 \to \mathbb C P^2$ is the inclusion of the top cell and $h$ the Hopf map. Then the homotopy $g_t^{(2)}$ extends to the desired homotopy $g_t \colon \mathbb C P^2 \times I \to BSG(4)$ if and only if the composite 
$$ (g_0 \cup g_1 \cup g_t^{(2)} )\circ m \colon S^4 \to X= \mathbb{C}P^2\times\partial I\cup S^2\times I \to BSG(4)$$
is null-homotopic. By standard algebraic topology one shows that $\pi_4(X)$ is isomorphic to $\mathbb Z$, generated by $[m]$. Denote $g_0 \cup g_1 \cup g_t^{(2)}$ by $\phi$, it suffices to show that $\phi_* \colon \pi_4(X) \to \pi_4(BSG(4))$ is trivial. 

By construction, the composite $j_* \circ \phi_* \colon \pi_4(X) \to \pi_4(BSG(4)) \to \pi_4(BSG)$ is trivial. On the other hand, the homomorphism $j_* \colon \pi_4(BSG(4)) \to \pi_4(BSG)$ can be identified with $\mathbb Z_{12} \oplus \mathbb Z \to \mathbb Z_{24}$, where on the first summand it is the inclusion. 
Therefore $\phi_*([m])$ is  not in the $\mathbb Z_{12}$-summand. In the sequel we will show it is zero by showing that it is zero in $\pi_4(BSG(4) )\otimes \mathbb Q$.

Let $X \to S^4 \vee X \vee S^4$ be the map which pinches a $4$-sphere from each top cell $D^4$ of $\mathbb C P^2$, and $S^4 \to  S^4 \vee S^4 \vee S^4$ be the map which pinches a $4$-sphere from each $D^4$ in $S^4 = D^4 \times \partial I \cup S^3 \times I$. 
Then the map $\phi \circ m \colon X \to BSG(4)$ has a factorization shown in the following commutative diagram
$$\xymatrix{
S^4 \ar[rr]^m \ar[d] &&  X \ar[r]^{\phi} \ar[d] & BSG(4) \\
S^4 \vee S^4 \vee S^4 \ar[rr]^{\mathrm{id} \vee m \vee \mathrm{-id} } && S^4 \vee X \vee S^4 \ar[ur]_{f_0 \vee f \vee f_1} & }$$
where $f_0=i((k+6j)\alpha+l\beta)$, $f_1=i(k\alpha+l\beta)$, $f=i(\eta\oplus \varepsilon^2)\times\partial I\cup g_t^{(2)}$.

First recall that we have shown that $[f_0] = [f_1] \in \pi_4(BSG(4))\otimes \mathbb{Q}$.

The only non-trivial rational homotopy group of $BSG(4)$ is $\pi_4(BSG(4))\otimes \mathbb{Q}\cong \mathbb{Q}$. There is a rational homotopy equivalence 
$l:BSG(4)\rightarrow K(\mathbb{Q},4)$ \cite{sullivan2005geometric}.  
We claim that the composite $l \circ f \colon X \to BSG(4) \to K(\mathbb Q,4)$ is null homotopic. This will finish the proof for the nonspin case. To prove the claim, first note that the inclusion $\mathbb C P^2 \cup \mathbb C P^2 \subset X$ induces an isomorphism $H^4(X;\mathbb Q) \to H^4(\mathbb C P^2; \mathbb Q) \oplus H^4( \mathbb C P^2 ; \mathbb Q)$. Therefore from the one-one correspondence between $H^4(-;\mathbb Q)$ and $[-,K(\mathbb Q,4)]$ it suffices to show that the map $l \circ i \circ (\eta \oplus \varepsilon^2) \colon \mathbb C P^2 \to BSG(4) \to K(\mathbb Q, 4)$ is null-homotopic. Notice that the map $\eta \oplus \varepsilon^2$ factors through $BSO(3)$.  Since the composition $\pi_3(SO(3))\rightarrow \pi_3(SO(4))\rightarrow \pi_3(SG(4))$ is trivial after tensoring with $\mathbb{Q}$, the map $BSO(3)\rightarrow BSO(4)\rightarrow BSG(4)\rightarrow K(\mathbb{Q},4)$ is trivial. This shows the claim.

\end{proof}

\section{surgery exact sequence and homotopy classification} \label{4}
In this section we prove the homotopy classification (Theorem \ref{homotopy}) by showing that  every element in the PL structure set $\mathscr{S}^{PL}(M_{k,l})$ is represented by a fiber homotopy equivalence $f_j \colon M_{k+6j,l} \to M_{k,j}$. The same also holds for the nonspin case. This approach was first used in \cite{crowley}.

For the convenience of the reader, we give a briefy introduction to simply connected surgery exact sequence. We work in the piecewise linear (PL) category. For more information we refer to \cite{surgery}.

Let $(M,\partial M)$ be an $n$-dimensional compact PL manifold with possibly empty boundary. A PL manifold structure on $(M,\partial M)$ is a triple $(N,\partial N,g)$, 
where $N$ is a compact PL manifold and  $g:(N,\partial N) \rightarrow (M,\partial M)$ is a homotopy equivalence. 
Two triples $(N_0,\partial N_0,g_0)$ and $(N_1,\partial N_1, g_1)$ are concordant if there is a PL-homeomorphism
$f:(N_0,\partial N_0)\rightarrow (N_1,\partial N_1)$ so that the following diagram commutes up to homotopy.

\begin{center}
    \begin{tikzcd}
        (N_0,\partial N_0)\ar[dd,"f"']\drar["g_0"] &\\
        &(M,\partial M)\\
        (N_1,\partial N_1)\urar["g_1"']&
    \end{tikzcd}
\end{center}
The PL structure set $\mathscr{S}^{PL}(M)$ is the set of concordance classes of triples $(N,\partial N, g)$. 

Let $G/PL$ be the homotopy fiber of $J \colon BPL\rightarrow BG$. Then the set of homotopy classes of maps $M\rightarrow G/PL$ are in 1-1 correspondence with the set 
of equivalence classes of pairs $(\eta,t)$, where $\eta$ is a stable PL bundle over $M$ and $t$ is a fiber homotopy trivialization of the associated
spherical fibration $J\eta$.

The map of normal invariants $\eta: \mathscr{S}^{PL}(M)\rightarrow [M,G/PL]$ is defined by sending a triple $(N,\partial N, g)$ to $\eta(g):=(\nu_M-(g^{-1})^*\nu_N,t(g))$,
where $t(g)$ is a fiber homotopy trivialization of $J(\nu_M-(g^{-1})^*\nu_N)$ defined in a standard way.

If $\pi_1(M)\cong \pi_1(\partial M) \cong 0$ and $n\geq 6$, we have the following two facts \cite{Sullivan1996}:
\begin{enumerate}
    \item If $\partial M\neq \emptyset$, $\mathscr{S}^{PL}(M)\xrightarrow{\eta} [M,G/PL]$ is an isomorphism;
    \item If $\partial M = \emptyset$, we have an exact sequence of based sets
        $$0\rightarrow \mathscr{S}^{PL}(M)\xrightarrow{\eta} [M,G/PL] \rightarrow L_{n}(e).$$
\end{enumerate}
where $L_n(e)\cong\pi_n(G/PL)\cong\left\{
    \begin{aligned}
        &\mathbb{Z}, &\mathrm{if}\ n=0\ \mathrm{ \mathrm{mod}}\ 4,\\
        &\mathbb{Z}_2,&\mathrm{if}\ n=2\ \mathrm{ \mathrm{mod}}\ 4,\\
        &0, &\mathrm{if}\ n=1\ \mathrm{ \mathrm{mod}}\ 2.
    \end{aligned}
    \right.$
\vspace*{10pt}

Let $\iota:M_{k,l} \hookrightarrow W_{k,l}$ be the canonical inclusion. Then 
we have the following commutative diagram, where the normal invariants are isomorphisms by the two facts mentioned above. 
\begin{figure}[htbp]
\centering
\begin{tikzcd}
    \mathscr{S}^{PL}(W_{k,l})\dar\rar["\eta","\cong"'] &\left[W_{k,l},G/PL\right]\dar\\
    \mathscr{S}^{PL}(M_{k,l})\rar["\eta","\cong"'] &\left[M_{k,l},G/PL\right]
\end{tikzcd}
\caption{Diagram 3}
\label{diagram_3}
\end{figure}

We have the same diagram when $M_{k,l}$ and $W_{k,l}$ are replaced by $M^{\prime}_{k,l}$ and $W^{\prime}_{k,l}$, respectively.

The infinite loop space structure on $G/PL$ defines a generalized cohomology theory, in which $\left[-,G/PL\right]$ is the $0$-th group.
Let $M_{k,l}^0$ denote $M_{k,l} \setminus \{ \mathrm{pt} \}$.
We have $\left[M_{k,l}^0,G/PL\right]\cong \mathscr{S}^{PL}(M_{k,l})$ \cite{Sullivan1996}. 
Then by diagram \ref{diagram_3}, we have $\left[M_{k,l},G/PL\right]\cong \left[M_{k,l}^0,G/PL\right]$.
It is easy to see that the Atiyah-Hurzebruch spectral sequence for $M_{k,l}^0$ collapses at the $E_2$-page, then so does $M_{k,l}$.
It's ready to see that the Atiyah-Hurzebruch spectral sequences for both $\mathbb{C}P^2$ and $M_{k,l}$ collapse at the $E_2$-page, therefore we have the following commutative diagram
    \begin{center}
    \begin{tikzcd}
        0   \rar    &\mathbb{Z}     \dar[two heads] \rar &  \left[\mathbb{C}P^2,G/PL\right] \rar\dar["(p_{k,l})^*"] &\mathbb{Z}_2 \rar\dar["\cong"] &0\\
        0   \rar    &\mathbb{Z}_l      \rar &  \left[M_{k,l},G/PL\right] \rar &\mathbb{Z}_2 \rar &0
    \end{tikzcd}
    \end{center}
where $p_{k,l}$ denotes the sphere bundle projection. 
By the $5$-lemma, $(p_{k,l})^*:\left[\mathbb{C}P^2,G/PL\right] \rightarrow \left[M_{k,l},G/PL\right]$ is surjective. It's known that $\left[\mathbb{C}P^2,G/PL\right]\cong \mathbb{Z}$ (\cite[Chapter 14C]{Wall1999Surgery}) , 
therefore $\left[M_{k,l},G/PL\right]\cong \mathbb{Z}_{2l}$. Similarly, we have $\left[M^{\prime}_{k,l},G/PL\right]\cong \mathbb{Z}_{2l}$.

In Lemma \ref{fiber} we have shown that there are fiber homotopy equivalences $f_j \colon M_{k+6j,l} \to M_{k,l}$ and $f'_j \colon M'_{k+6j,l} \to M'_{k,l}$. They represent elements in $\mathscr{S}^{PL}(M_{k,l})$ and $\mathscr{S}^{PL}(M'_{k,l})$, respectively. 

\begin{lemma}\label{invariant}
Let $f_j \colon M_{k+6j,l} \to M_{k,l}$ be a fiber homotopy equivalence, then $\eta(f_j)=j\in [M_{k,l},G/PL]\cong\mathbb{Z}_{2l}$. The same holds when $M_{k,l}$ and $f_j$
are replaced by $M_{k,l}^{\prime}$ and $f_j^{\prime}$, respectively.
\end{lemma}
\begin{proof}
    For any vector bundle $\zeta$, the disc bundle $D(\zeta)$ is homeomorphic to the mapping cylinder of spherical bundle projection.
Therefore, there exists a natural fiber homotopy equivalence $F_j:W_{k+6j,l}\rightarrow W_{k,l}$ extending $f_j$. Consider the following commutative diagram

\begin{center}
\begin{tikzcd}
    \mathcal{S}^{PL}(W_{k,l}) \rar["\eta"]\dar["i_{k,l}^*"] & \left[W_{k,l},G/PL\right] \rar["j_*"]\dar["i_{k,l}^*"] & \left[W_{k,l},BPL\right] \dar["i_{k,l}^*"] \\
    \mathcal{S}^{PL}(\mathbb{C}P^2) \rar["\eta"] & \left[\mathbb{C}P^2,G/PL\right] \rar["j_*","\times 24"'] & \left[\mathbb{C}P^2,BPL\right] 
\end{tikzcd}
\end{center}
where $i_{k,l}:\mathbb{C}P^2\rightarrow W_{k,l}$ is the inclusion of the zero section and $j_*$ is interpreted by forgetting the trivialization.
It's well know that $[\mathbb{C}P^2,BPL]\cong \mathbb{Z}$ and this isomorphism is determined by $p_1$. 
And $j_*:[\mathbb{C}P^2,G/PL]\rightarrow [\mathbb{C}P^2,BPL]$ is known to be multiplication by $24$ \cite{wxq}.
Thus we obtain
\begin{align*}
    i_{k,l}^*j_*\eta(F_j)&=i_{k,l}^*(\nu(W_{k,l})-F_j^{-1*}\nu(W_{k+6j,l}))\\
            &=i_{k,l}^*(\pi^{*}_{k,l}(-\xi_{k,l}-\tau(\mathbb{C}P^2))-F_j^{-1*}\pi^{*}_{k+6j,l}(-\xi_{k+6j,l}-\tau(\mathbb{C}P^2)))\\
            &=i_{k,l}^*\pi^{*}_{k,l}(\xi_{k+6j,l}-\xi_{k,l})\\
            &=\xi_{k+6j,l}-\xi_{k,l}\\
            &=4(k+6j)-2l-(4k-2l)\\
            &=24j\in \mathbb{Z}.
\end{align*}
Hence $\eta(F_j)=j\in \left[W_{k,l},G/PL\right]\cong \mathbb{Z}$. From the following commutative diagram
\begin{center}
\begin{tikzcd}[column sep=large,row sep =large]
    \left[W_{k,l},G/PL\right] \dar["\iota^*"'] & \left[\mathbb{C}P^2,G/PL\right] \lar["\cong","\pi^{*}_{k,l}"']\dlar["p^{*}_{k,l}"]\\
    \left[M_{k,l},G/PL\right]   &
\end{tikzcd}
\end{center}
we conclude that $\eta(f_j)=\iota^*\eta(F_j)=p^{*}_{k,l}\pi^{*-1}_{k,l}\eta(F_j)=j\in \mathbb{Z}_{2l}$.

The proof for $M^{\prime}_{k,l}$ is similar.
\end{proof}

\begin{proposition}\label{structure_set}
    A  $7$-dimensional closed nonspin (spin) PL-manifold  $M$ is homotopy equivalent to $M_{k,l}$ ($M^{\prime}_{k,l}$) if and only if $M$ is 
    PL-homeomorphism to $M_{k+6j,l}$ ($M^{\prime}_{k+6j,l}$) for some $j\in \mathbb{Z}$.
\end{proposition}
\begin{proof}
    Let $f:M\rightarrow M_{k,l}$ be a homotopy equivalence. Then by Lemma \ref{invariant}, there exists some $j$ such that
    $\eta(f_j)=\eta(f)$. Since $\eta \colon \mathscr{S}^{PL}(M_{k,l}) \to [M_{k,l},G/PL]$ is an isomorphism (see diagram \ref{diagram_3}), one has $(M_{k+6j,l},f_j) = (M,f) \in \mathscr{S}^{PL}(M_{k,l})$. Hence $M_{k+6j,l}$ and $M$ are 
     PL-homeomorphic. 
\end{proof}

Theorem \ref{homotopy} is a consequence of Proposition \ref{structure_set} and the PL classification. 
\begin{proof}[Proof of Theorem \ref{homotopy}]
    Suppose $M_{k,l}$ is homotopy equivalent to $M_{k^{\prime},l}$, then by Proposition \ref{structure_set} there exists some $j$ 
    such that $M_{k,l}$ is PL-hemeomorphic to $M_{k^{\prime}+6j,l}$. 
    Then Theorem \ref{result} (5) implies  $k \equiv k^{\prime}+6j \pmod 6$. Hence $k \equiv k^{\prime} \pmod 6$.
    The proof is analogous for $M^{\prime}_{k,l}$.
\end{proof}

\begin{remark}
    Combining Theorem \ref{homotopy} and Lemma \ref{fiber},
    we obtain that the homotopy classification of the total spaces $M_{k,l}$  and the homotopy classification of the 
    pairs $(M_{k,l},S^3)$  coincide. 
    The statemen holds when $M_{k,l}$ and $(M_{k,l},S^3)$ are replaced by $M_{k,l}^{\prime}$ and $(M^{\prime}_{k,l},S^3)$, respectively.
\end{remark}


\noindent \textbf{Acknowledgement.}
The author wishes to thank Yang Su for his guidance and many constructive discussions.
\bibliography{reference.bib}
\bibliographystyle{unsrt}

\end{document}